\documentclass{amsart}
\usepackage{amsmath, amssymb, amsthm}
\usepackage[margin=1in]{geometry}
\usepackage[dvipdfmx]{hyperref}

  \makeatletter
    
    \@addtoreset{equation}{section}
  \makeatother
  
\newcommand{\supp}{\mathrm{supp}\hspace{0.5mm}}

\newtheorem{theorem}{Theorem}[section]
\newtheorem{lemma}{Lemma}[section]
\newtheorem{proposition}{Proposition}[section]
\newtheorem*{theoremA}{Theorem A}

\begin{document}

\title[A remark on the Schr\"odinger operator on Wiener amalgam spaces]
{A remark on the Schr\"odinger operator on Wiener amalgam spaces}

\author[T. Kato]{Tomoya Kato}
\author[N. Tomita]{Naohito Tomita}

\address{Department of Mathematics, Graduate School of Science, Osaka University, Toyonaka, Osaka 560-0043, Japan}
\email[Tomoya Kato]{t.katou@cr.math.sci.osaka-u.ac.jp}
\email[Naohito Tomita]{tomita@math.sci.osaka-u.ac.jp}


\begin{abstract}
In this paper, we study the boundedness of the Schr\"odinger operator $e^{i \Delta}$ on Wiener amalgam spaces and determine its optimal condition. 
\end{abstract}

\maketitle

\section{Introduction} \label{sec1}

In this paper, we consider the Schr\"odinger operator $e^{i \Delta}$ 
which is naturally arose from the Cauchy problem for the free Schr\"odinger equation
\begin{equation*}
\left\{ 
\begin{array}{ll}
i\partial_t u + \Delta u = 0, & t \in \mathbb{R},\ x \in \mathbb{R}^n, \\
u( 0, x )=u_0 (x),  &  x \in  \mathbb{R}^n.
\end{array} 
\right.
\end{equation*}
Here, the differential operator $\Delta$ is the usual Laplacian. 
The Schr\"odinger operator is a special case of Fourier multipliers
and denoted by
\begin{equation} \label{sch op}
e^{i \Delta} f (x) = (2\pi)^{-n} \int_{ \mathbb{R}^n } e^{ i x \cdot \xi } e^{ - i | \xi |^2 } \widehat f (\xi) d \xi.
\end{equation}

Wiener amalgam spaces $W_{p,q}^s$ are variation of modulation spaces $M_{p,q}^s$ 
which were introduced  from the view point of time-frequency analysis 
(see Feichtinger \cite{feichtinger 1983} and Gr\"ochenig \cite{grochenig 2001}). 
The essence of modulation and Wiener amalgam spaces is
to measure integrability (or decay property) of functions 
with respect to their space variable and frequency variable simultaneously. 
On the other hand, in $L^p$-Sobolev spaces $L^p_s$ and Besov spaces $B^s_{p,q}$, 
the space and frequency variables of functions must be considered separately.
This idea is a major difference among these spaces. 
The exact definitions and basic properties of modulation and Wiener amalgam spaces
will be given in Section \ref{sec2.2}. 
In the following, if $s = 0$ we simply write $ M_{p,q} $ and $ W_{p,q} $
instead of $ M_{p,q}^s $ and $ W_{p,q}^s $, respectively.

The boundedness problems of Fourier multipliers have been studied by many researchers. 
In particular, we here focus on the unimodular Fourier multiplier 
$e^{- i |D|^\alpha}$ for $\alpha \geq 0$,
which is denoted by the expression (\ref{sch op}) 
whose phase function $| \xi |^2$ is replaced by $| \xi |^\alpha$. 
In \cite{hormander 1960}, H\"ormander studied  
that $e^{- i |D|^2}$ (i.e., $e^{i\Delta}$) is bounded on $L^p$ if and only if $p = 2$.
After that, it was proved that 
for $1 < p <\infty$, $e^{- i |D|^\alpha}$ is bounded from $ L^p_s $ to $L^p$ 
if and only if $s \geq \alpha n | 1/p - 1/2 | $ (see, e.g., Miyachi \cite{miyachi 1980}).
Next, we shall consider the boundedness results on modulation spaces.
On modulation spaces, Gr\"ochenig and Heil \cite{grochenig heil 1999} invented 
that $e^{- i |D|^2}$ is bounded on $M_{p,q}$ for all $1 \leq p,q \leq \infty$
(see also Toft \cite{toft 2004} and Wang, Zhao and Guo \cite{wang zhao guo 2006}).
Then, the boundedness of $e^{-i |D|^\alpha}$ on modulation spaces
was studied for $0 \leq \alpha \leq 2$ 
by B\'enyi, Gr\"ochenig, Okoudjou and Rogers \cite{benyi grochenig okoudjou rogers 2007}, 
and for $\alpha > 2$ 
by Miyachi, Nicola, Rivetti, Tabacco and Tomita \cite{miyachi nicola rivetti tabacco tomita 2009}.
Optimality of these results can be found in, e.g., 
\cite[Theorem 1.2]{miyachi nicola rivetti tabacco tomita 2009} 
and \cite[Remark 3.4]{tomita 2010}.
Extracting the Schr\"odinger operator case, namely, $\alpha=2$, we have the following:

\begin{theoremA}
Let $s \in \mathbb{R}$ and $1 \leq p,q \leq \infty$. Then, $e^{ i \Delta}$ is bounded from $M^s_{p,q}$ to $ M_{p,q} $ if and only if $s \geq 0$.
\end{theoremA}

Comparing the results on $L^p$-Sobolev and modulation spaces,
we see that the Schr\"odinger operator is bounded on modulation spaces without loss of regularity,
while on $L^p$-Sobolev spaces loss of the order up to $2n | 1/p - 1/2 | $ occurs.
This is one of advantages in using modulation spaces.

The objective of this paper is to study the boundedness problem of the Schr\"odinger operator $e^{ i \Delta}$ 
on variation of modulation spaces, Wiener amalgam spaces. 
In \cite[Theorem 1.1]{cunanan sugimoto 2014}, Cunanan and Sugimoto proved that 
for $1 \leq p,q \leq \infty$, $e^{ i \Delta}$ is bounded from $W^s_{p,q}$ to $W_{p,q} $ if $s > n | 1/p - 1/q |$.
Conversely, in Cordero and Nicola \cite[Proposition 6.1]{cordero nicola 2010},
it was shown that
$e^{ i \Delta}$ is bounded from $W^s_{p,q}$ to $W_{p,q} $ only if $s \geq n | 1/p - 1/q |$.
Moreover, in \cite[Proposition 6.1]{cordero nicola 2010}, 
it was also stated that if $p = \infty$ and $q < \infty$, 
the necessary condition for the boundedness holds with the strict inequality,
namely, $e^{ i \Delta}$ is bounded from $W^s_{\infty,q}$ to $W_{\infty,q} $ only if $s > n/q $.
Therefore, except for the case $p = \infty$ and $q < \infty$, there is a gap
between the sufficient and necessary conditions for this boundedness problem.
That is, the critical case $s = n | 1/p - 1/q | $ is not included in the sufficient condition,
although it is included in the necessary condition. 
The goal of this paper is to determine whether this critical case is needed or not.
The exact answer is provided by the following.

\begin{theorem} \label{main thm}
Let $s \in \mathbb{R}$ and $1 \leq p,q \leq \infty$. 
\begin{enumerate} 
\item For $p = q$, $e^{i\Delta}$ is bounded from $W^s_{p,q}$ to $W_{p,q} $ if and only if $s \geq 0$.
\item For $p \neq q$, $e^{i\Delta}$ is bounded from $W^s_{p,q}$ to $W_{p,q} $  if and only if $s > n | 1/p - 1/q |$.
\end{enumerate} 
\end{theorem}

Theorem \ref{main thm} mentions that unlike modulation spaces,
$e^{i \Delta}$ is not bounded on $W_{p,q} $ if $p \neq q$
(Wiener amalgam spaces coincide with modulation spaces in the case $p = q$).
However, the boundedness of unimodular Fourier multiplier $e^{-i |D|^\alpha}$ on $W_{p,q}$ 
for $0 \leq \alpha \leq 1$ was given 
by Cunanan and Sugimoto \cite[Corollary 2.1]{cunanan sugimoto 2014}.
Since $e^{-i |D|^\alpha} = e^{i \Delta}$ if $\alpha = 2$,
we have yet to determine an upper bound of $\alpha$ 
to obtain the boundedness of $e^{- i |D|^\alpha}$ on $W_{p,q} $.

We end this section with mentioning the plan of this paper. 
In Section \ref{sec2}, we will state basic notations which will be used throughout this paper, 
and then introduce the definitions and some basic properties of modulation and Wiener amalgam spaces.
After stating and proving some lemmas needed to show our main theorem in Section \ref{sec3}, 
we will actually prove it in Section \ref{sec4}.

\section{Preliminaries}\label{sec2}

\subsection{Basic notations} \label{sec2.1}
We collect notations which will be used throughout this paper.
We denote by $\mathbb{R}$ and $\mathbb{Z}$
the sets of reals and integers, respectively. 
The notation $a \lesssim b$ means $a \leq C b$ with a constant $C > 0$ which may be different in each occasion, and $a \sim b $ means $a \lesssim b$ and $b \lesssim a$. We write $\langle x \rangle = (1 + | x |^2)^{1/2}$.

We denote the Schwartz space of rapidly decreasing smooth functions
on $\mathbb{R}^n$ by $\mathcal{S} = \mathcal{S} (\mathbb{R}^n)$ and its dual,
the space of tempered distributions, by $\mathcal{S}^\prime = \mathcal{S}^\prime(\mathbb{R}^n)$. 
The Fourier transform and the inverse Fourier transform of $f \in \mathcal{S}(\mathbb{R}^n)$ are given by
\begin{equation*}
\mathcal{F} f  (\xi) = \widehat {f} (\xi) = \int_{\mathbb{R}^n}  e^{-i \xi \cdot x } f(x) d x
\quad\textrm{and}\quad
\mathcal{F}^{-1} f (x) = (2\pi)^{-n} \int_{\mathbb{R}^n}  e^{i x \cdot \xi } f( \xi ) d\xi,
\end{equation*}
respectively.
For $g \in \mathcal{S}^\prime (\mathbb{R}^n)$, 
the Fourier multiplier is denoted by
$g(D) f = \mathcal{F}^{-1} \left[g \cdot \mathcal{F} f \right]$,
and for $s \in \mathbb{R}$, the Bessel potential by
$(I - \Delta )^{s/2} f = \mathcal{F}^{-1} [ \langle \cdot \rangle^s \cdot \mathcal{F}f ]$ 
for $f \in \mathcal{S}(\mathbb{R}^n)$.

We will use some function spaces.
The Lebesgue space $L^p = L^p(\mathbb{R}^n) $ is equipped with the norm 
\begin{equation*}
\| f \|_{L^p} = \left( \int_{\mathbb{R}^n} \big| f(x) \big|^p dx \right)^{1/p} 
\end{equation*}
for $1 \leq p < \infty$. If $p = \infty$, $\| f \|_{\infty} = \textrm{ess.}\sup_{x\in\mathbb{R}^n} |f(x)|$. 
Moreover, we denote the $L^p$-Sobolev space (or the Bessel potential space) $L^p_s$
by $L^p_s (\mathbb{R}^n) = \{ f \in \mathcal{S}^\prime (\mathbb{R}^n) : \| (I-\Delta)^{s/2} f \|_{L^p} < \infty \}$
for $1 < p < \infty$ and $s\in\mathbb{R}$.
For $1 \leq q \leq \infty$, we denote by $\ell^q$ the set of all complex number sequences 
$\{ a_k \}_{ k \in \mathbb{Z}^n }$ such that
\begin{equation*}
 \| a_k \|_{ \ell^q } = \left( \sum_{ k \in \mathbb{Z}^n }  | a_k  |^q \right)^{ 1 / q } < \infty,
\end{equation*}
with the usual modification for $ q = \infty$. 
For the sake of simplicity, we write $ \| a_k  \|_{ \ell^q } $ 
instead of the more correct notation $ \| \{ a_k \}_{ k\in\mathbb{Z}^n } \|_{ \ell^q } $.

\subsection{Modulation and Wiener amalgam spaces} \label{sec2.2}
We give the definitions of modulation and Wiener amalgam spaces. 
They are based on Feichtinger \cite{feichtinger 1983} and Gr\"ochenig \cite{grochenig 2001}. 
We fix a function (called a window function) $g \in \mathcal{S} (\mathbb{R}^n) \setminus \{0\}$ 
and denote the short-time Fourier transform of $f \in \mathcal{S}^\prime (\mathbb{R}^n) $ with respect to $g$ by 
\begin{equation*}
V_g f (x,\xi) = \int_{\mathbb{R}^n} e^{ - i \xi \cdot t } \ \overline {g (t-x)} f(t) d t.
\end{equation*}
We will sometimes write $V_g [ f ] $ when the function $f$ is complicated. 
Now, for $s \in \mathbb{R}$ and $1 \leq p,q \leq \infty$, the modulation space $M_{ p , q }^s$ is denoted by
\begin{equation*}
M_{p,q}^s (\mathbb{R}^n)
= 
\left\{ f\in\mathcal{S}^\prime(\mathbb{R}^n) : \left\| f \right\|_{M_{p,q}^s} = \left\| \left\| \langle \xi \rangle^s V_g f (x,\xi) \right\|_{L^p(\mathbb{R}^n_x)} \right\|_{L^q(\mathbb{R}^n_\xi)} < +\infty \right\},
\end{equation*}
and the Wiener amalgam space $W_{ p , q }^s$ is denoted by
\begin{equation*}
W_{p,q}^s (\mathbb{R}^n)
= 
\left\{ f\in\mathcal{S}^\prime(\mathbb{R}^n) : \left\| f \right\|_{W_{p,q}^s} = \left\| \left\| \langle \xi \rangle^s V_g f (x,\xi) \right\|_{L^q(\mathbb{R}^n_\xi)} \right\|_{L^p(\mathbb{R}^n_x)} < +\infty \right\}.
\end{equation*}
These definitions are independent of the choices of the window function $g$. 
From the definitions, we see that $M_{p,p}^s = W_{p,p}^s $. 
We also remark that $M_{2,2}^s = W_{2,2}^s = L^2_s $.
In the following, if $s = 0$ we will simply write $ M_{p,q} $ and $ W_{p,q} $
instead of $ M_{p,q}^s $ and $ W_{p,q}^s $, respectively.

We write $ X_{p,q}^s $ instead of $ M_{p,q}^s $ or $ W_{p,q}^s $. 
We note that $X_{p,q}^s$ are Banach spaces
and $\mathcal{S} \subset X_{p,q}^s \subset \mathcal{S}^\prime$. 
In particular, $\mathcal{S}$ is dense in $X_{p,q}^s $ if $1 \leq p,q < \infty$. 
For $1 \leq p,q < \infty$, 
the dual space of $ X_{p,q}^s $ can be seen as $\big( X_{p,q}^s \big)^\prime =  X_{p^\prime , q^\prime  }^{-s} $.
Moreover, we have the following complex interpolation theorem. 
If $0 < \theta < 1$, $s = (1-\theta)s_1 + \theta s_2$, $ 1/p= (1-\theta)/p_1 + \theta/p_2 $ and $ 1/q= (1-\theta)/q_1 + \theta/q_2$, 
we have $\big( X_{p_1,q_1}^{s_1},X_{p_2,q_2}^{s_2} \big)_\theta = X_{p,q}^{s}$.

We finally present the basic inclusion relations and the lifting property.

\begin{proposition} \label{basic emb}
Let $s_1, s_2 \in \mathbb{R}$ and $1 \leq p, p_1, p_2, q_1, q_2 \leq \infty$.
\begin{enumerate} 
\item $X_{p_1,q_1}^{s_1} \hookrightarrow X_{p_2,q_2}^{s_2}$ if $p_1 \leq p_2$, $q_1 \leq q_2$ and $s_1 \geq s_2$.
\item $X_{p,q_1}^{s_1} \hookrightarrow X_{p,q_2}^{s_2} $ if $q_1 > q_2$ and $s_1 - s_2 > n ( 1 / q_2 - 1/q_1 )$.
\end{enumerate} 
\end{proposition}

\begin{proposition} \label{lift op}
Let $s,t\in\mathbb{R}$ and $1 \leq p,q \leq \infty$. 
Then, $(I-\Delta)^{t/2}$ maps $ X_{p,q}^s $ isomorphically onto $X_{p,q}^{s-t}$,
and $\left\| (I-\Delta)^{t/2} f \right\|_{ X_{p,q}^{s-t} } \sim \left\| f \right\|_{ X_{p,q}^s }$. 
\end{proposition}

\section{Lemma}\label{sec3}

In this section, we state a key lemma and its proof,
which will be used for the proof of the ``ONLY IF'' part of Theorem \ref{main thm} (2).

\begin{lemma} \label{sequence}
Let $s \in \mathbb{R}$ and $1 \leq p,q < \infty$.
If $e^{ i \Delta}$ or $e^{ - i \Delta}$ is bounded from $W^s_{p,q}$ to $W_{p,q}$, then
\begin{equation} \label{seq ck}
\left( \sum_{k\in\mathbb{Z}^n}  | c_k |^p \right)^{1/p} 
\lesssim \left( \sum_{k\in\mathbb{Z}^n} \langle k \rangle^{sq} | c_k |^q \right)^{1/q}
\end{equation}
holds for all finitely supported sequences $\{ c_k \}_{k\in\mathbb{Z}^n}$ 
(that is, $c_k = 0$ except for a finite number of k's).
\end{lemma}

\begin{proof} [\bf Proof of Lemma \ref{sequence}]
Before beginning with the proof, we give the following notations which will be used here:
\begin{align*}
\| f \|_{ M_{p,q}^{(s_1, s_2)} } 
&= 
\left\| 
	\left\| \langle x \rangle^{s_1} \langle \xi \rangle^{s_2} V_g f (x,\xi) 
	\right\|_{L^p (\mathbb{R}^n_x)}
\right\|_{L^q (\mathbb{R}^n_\xi)} ;
\\
\| f \|_{ W_{p,q}^{(s_1, s_2)} } 
&= 
\left\| 
	\left\| \langle x \rangle^{s_1} \langle \xi \rangle^{s_2} V_g f (x,\xi) 
	\right\|_{L^q (\mathbb{R}^n_\xi)} 
\right\|_{L^p (\mathbb{R}^n_x)}.
\end{align*}
Note that $\| f \|_{ M_{p,q}^s } = \| f \|_{ M_{p,q}^{(0, s)} } $ 
and $\| f \|_{ W_{p,q}^s } = \| f \|_{ W_{p,q}^{(0, s)} }$.
From the fact $| V_g f (x,\xi) | = (2\pi)^{-n} | V_{\widehat g} \widehat f (\xi,-x) |$ 
(see \cite[Section 2.2]{grochenig heil 1999} for the detail calculation), 
$ \| f \|_{ W_{p,q}^{(s_1, s_2)} } \sim \| \widehat f \|_{ M_{q,p}^{(s_2, s_1)} } $ for $1 \leq p,q \leq \infty$. 
Hence, the assumption $\| e^{ \pm i \Delta} f \|_{ W_{p,q} } \lesssim \| f \|_{ W_{p,q}^{s} } $
is equivalent to the inequality 
\begin{equation} \label{equiv ass}
\left\| e^{ \mp i | \cdot |^2 } \widehat f \right\|_{ M_{q,p}^{(0, 0)} } 
\lesssim \left\| \widehat f \right\|_{ M_{q,p}^{(s, 0)} } .
\end{equation} 
In the following statements, we will consider (\ref{equiv ass}) and write $f$ instead of $\widehat f$ for the sake of simplicity.

\noindent{\bf Step 1:} 
We first consider the left hand side of (\ref{equiv ass}). 
We compute the short-time Fourier transform of $e^{ \pm i | \cdot |^2 } f$.
\begin{align*}
\left| V_g [e^{ \pm i | \cdot |^2 } f] (x,\xi) \right| 
&= 
\left| 
	\int_{ \mathbb{R}^n } e^{-i \xi \cdot t } 
	\
	\overline {g ( t - x )} \cdot e^{ \pm i | t |^2 } f ( t ) dt 
\right| 
\\
&=
\left| 
	\int_{ \mathbb{R}^n } e^{-i (\xi \mp 2x ) \cdot t } 
	\
	\overline {g ( t - x )} \cdot e^{ \pm i | t - x |^2 } f ( t ) dt 
\right| 
\\
&= 
\left| V_{g \cdot e^{ \mp i | \cdot |^2 }} f (x, \xi \mp 2x ) \right| ,
\end{align*}
where it should be remarked that 
$g \cdot e^{ \mp i | \cdot |^2 } \in \mathcal{S}(\mathbb{R}^n) \setminus \{ 0 \}$ 
since $e^{ \mp i | \cdot |^2 } $ and all their derivatives are $C^\infty$ slowly increasing functions.
Since the norm of modulation spaces is independent of the choice of window functions,
\begin{align} 
\label{equiv scho mod}
\begin{split}
\left\| e^{ \pm i | \cdot |^2 } f \right\|_{ M_{q,p}^{(0, 0)} } 
&\sim 
\left\| 
	\left\| V_{\psi \cdot e^{ \pm i | \cdot |^2 }} 
		\left[ e^{ \pm i | \cdot |^2 } f 
		\right]( x, \xi ) 
	\right\|_{L^q (\mathbb{R}^n_x)} 
\right\|_{L^p (\mathbb{R}^n_\xi)}
\\
&=
\left\| \left\| V_\psi f ( x, \xi \mp 2x ) \right\|_{L^q (\mathbb{R}^n_x)} \right\|_{L^p (\mathbb{R}^n_\xi)}
\end{split}
\end{align}
for any $\psi \in \mathcal{S}(\mathbb{R}^n) \setminus \{ 0 \}$. 

We choose $\varphi  \in \mathcal{S} (\mathbb{R}^n)$ satisfying that 
\begin{equation*}
\supp \varphi \subset [ - 1/8, 1/8 ]^n 
\quad\textrm{and}\quad 
| \widehat \varphi | \geq c >0 \textrm{ on } [ - 3/8, 3/8 ]^n
\end{equation*}
for some positive constant $c$, and $\psi \in \mathcal{S} (\mathbb{R}^n)$ satisfying that 
\begin{equation*}
\supp \psi \subset [ - 3/8, 3/8 ]^n
\quad\textrm{and}\quad 
\psi \equiv 1 \textrm{ on } [ -1/4, 1/4 ]^n
\end{equation*}
(see \cite[Lemma 4.3]{kobayashi miyachi tomita 2009} for existence of such functions $\varphi$ and $\psi$). We put
\begin{equation*}
f(t) = \sum_{\ell\in\mathbb{Z}^n}  c_\ell \cdot \varphi \left( t - \ell \right)
\end{equation*}
for a finitely supported sequence $\{c_\ell\}_{\ell\in\mathbb{Z}^n}$. 

Under the conditions above, 
we first consider the inner $L^q$ norm in the last quantity in (\ref{equiv scho mod}) and have
\begin{equation} \label{integ qk}
\left\| V_\psi f ( x, \xi \mp 2x ) \right\|_{L^q (\mathbb{R}^n_x)} 
\geq 
\left( \sum_{ k \in \mathbb{Z}^n } \int_{ Q_k } 
	\left| \int_{ \mathbb{R}^n } e^{-i (\xi \mp 2x) \cdot t} 
	\
	\overline{\psi (t-x)} \sum_{ \ell \in\mathbb{Z}^n } c_\ell 
		\cdot \varphi (t - \ell) dt 
	\right|^q dx
\right)^{1/q},
\end{equation}
where we set $ Q_k = k + [-1/8, 1/8]^n $. Considering the supports of the functions in the integral of (\ref{integ qk}), we have for all $x \in Q_k$
\begin{align*}
\supp \psi (\cdot-x) &\subset x + [ - 3/8, 3/8 ]^n \subset k + [ - 1/2, 1/2 ]^n ; \\
\supp \varphi (\cdot - \ell) &\subset \ell + [ - 1/8, 1/8 ]^n,
\end{align*}
so that we see that $\supp \psi (\cdot-x) \cap \supp \varphi (\cdot - \ell) = \varnothing$ if $k \neq \ell$ for all $x \in Q_k$. Hence, the right hand side of (\ref{integ qk}) is equal to
\begin{equation} 
\label{integ qk 2}
\left( \sum_{ k \in \mathbb{Z}^n } |c_k |^q \int_{ Q_k } 
	\left| \int_{ \mathbb{R}^n } e^{-i (\xi \mp 2x) \cdot t} \ \overline{\psi (t-x)} \cdot \varphi (t - k) dt 
	\right|^q dx 
\right)^{1/q}.
\end{equation}
Moreover, we realize that $\psi (\cdot-x) \equiv 1 $ on $\supp \varphi (\cdot - k )$, 
since we have the fact that $\supp \varphi (\cdot - k ) \subset k + [-1/8,1/8]^n \subset x + [-1/4,1/4]^n$ 
which is given from the equivalence $x \in Q_k \Leftrightarrow k \in Q_x$. 
Then, we have
\begin{align*}
(\ref{integ qk 2}) 
&= 
\left( \sum_{ k \in \mathbb{Z}^n } |c_k |^q \int_{ Q_k } 
	\left| \int_{ \mathbb{R}^n } e^{-i (\xi \mp 2x) \cdot t } \varphi (t - k) dt 
	\right|^q dx
\right)^{1/q} \\ 
&= 
\left( \sum_{ k \in \mathbb{Z}^n } |c_k |^q \int_{ Q_k } 
	\left| \widehat \varphi (\xi \mp 2x) 
	\right|^q dx 
\right)^{1/q}.
\end{align*}
Substituting this conclusion into the last quantity of (\ref{equiv scho mod}), we have
\begin{align} 
\label{integ qk 3}
\begin{split}
\left\| 
	\left\| 
		V_\psi f ( x, \xi \mp 2x ) 
	\right\|_{L^q (\mathbb{R}^n_x)} 
\right\|_{L^p (\mathbb{R}^n_\xi)} 
&\geq
\left\| 
	\left( \sum_{ k \in \mathbb{Z}^n } |c_k |^q \int_{ Q_k } 
		\left| 
		\widehat \varphi (\xi \mp 2x) 
		\right|^q dx 
	\right)^{1/q} 
\right\|_{L^p (\mathbb{R}^n_\xi)}
\\
&\geq
\left( \sum_{ m \in \mathbb{Z}^n} \int_{ Q_{\pm 2m } } 
	\left( \sum_{ k \in \mathbb{Z}^n } |c_k |^q \int_{ Q_k } 
		\left| 
		\widehat \varphi (\xi \mp 2x) 
		\right|^q dx 
	\right)^{p/q} d\xi 
\right)^{1/p} 
\\
&\geq
\left( \sum_{ m \in \mathbb{Z}^n} |c_m |^p \int_{ Q_{\pm 2m } } 
	\left( \int_{ Q_m } 
		\left| \widehat \varphi (\xi \mp 2x) 
		\right|^q dx 
	\right)^{p/q} d\xi 
\right)^{1/p}, 
\end{split}
\end{align}
where we set $Q_{\pm 2m } = \pm 2m + [-1/8,1/8]^n $. 
The assumption $| \widehat \varphi | \geq c >0 \textrm{ on } [ - 3/8, 3/8 ]^n$ 
and the fact $\xi \mp 2x \in ( \pm 2m + [-1/8,1/8]^n ) \mp 2 (m + [-1/8,1/8]^n) \subset [ - 3/8, 3/8 ]^n$ 
yield that the last expression in (\ref{integ qk 3}) is estimated from below by
\begin{align*} 
\left( \sum_{ m \in \mathbb{Z}^n} |c_m |^p \int_{ Q_{\pm 2m } } \left( \int_{ Q_m } c^q dx \right)^{p/q} d\xi \right)^{1/p},
\end{align*}
which gives
\begin{align*}
\left\| \left\| V_\psi f ( x, \xi \mp 2x ) \right\|_{L^q (\mathbb{R}^n_x)} \right\|_{L^p (\mathbb{R}^n_\xi)} 
\gtrsim
\left( \sum_{ m \in \mathbb{Z}^n} |c_m |^p \right)^{1/p} .
\end{align*}
Combining this result with (\ref{equiv scho mod}), we obtain
\begin{equation} \label{conc prop 1}
\left\| e^{ \pm i | \cdot |^2 } f \right\|_{ M_{q,p}^{(0, 0)} } 
\gtrsim 
\left( \sum_{ m \in \mathbb{Z}^n} |c_m |^p \right)^{1/p} .
\end{equation}

\noindent{\bf Step 2:} Next, we estimate the right hand side of (\ref{equiv ass}). 
The statement of this step is essentially the same as that of \cite[Theorem 12.2.4]{grochenig 2001}. 
Recall that we put
\begin{equation*}
f(t) = \sum_{\ell\in\mathbb{Z}^n}  c_\ell \cdot \varphi \left( t - \ell \right)
\end{equation*}
for a finitely supported sequence $\{c_\ell\}_{\ell\in\mathbb{Z}^n}$. 
We calculate the short-time Fourier transform of this function $f$ with the $x$-weight. 
Using the $N$-times integration by parts with respect to the $t$-variable, we have for any $N\geq 0$
\begin{align*}
\langle x \rangle^{s} 
\left| V_g f (x,\xi) \right|
&= 
\langle x \rangle^{s} 
\left| 
	\sum_{\ell\in\mathbb{Z}^n} c_\ell \int_{ \mathbb{R}^n } e^{ - i \xi \cdot t } 
	\
	\overline{g (t-x)}\cdot \varphi \left( t - \ell \right) dt 
\right|
\\
&\lesssim
\langle x \rangle^{s} \langle \xi \rangle^{-N} 
\sum_{\ell\in\mathbb{Z}^n} | c_\ell | 
	\sum_{\substack{ \beta \leq \alpha \\ | \alpha | \leq N } } 
	\int_{ \mathbb{R}^n } \left| ( \partial^{ \alpha - \beta } g ) (t-x) \right|  
	\cdot 
	\left| ( \partial^{ \beta } \varphi )\left( t - \ell \right) \right| dt.
\end{align*}
From the inequality 
$\langle x - \ell \rangle \lesssim \langle t - x \rangle \cdot \langle t - \ell \rangle$ 
and the Schwarz inequality, 
it follows that for any $M \geq 0$
\begin{align*}
&
\int_{ \mathbb{R}^n } \left| ( \partial^{ \alpha - \beta } g ) (t-x) \right|  \cdot  \left| ( \partial^{ \beta } \varphi )\left( t - \ell \right) \right| dt 
\\
&\lesssim
\langle x - \ell \rangle^{-M} 
\int_{ \mathbb{R}^n } 
\langle t - x \rangle^M \left| ( \partial^{ \alpha - \beta } g ) (t-x) \right| 
\cdot \langle t - \ell \rangle^{M}  \left| ( \partial^{ \beta } \varphi ) ( t - \ell ) \right| dt 
\\
&\leq
\langle x - \ell \rangle^{-M} 
\left\| \langle \cdot \rangle^M ( \partial^{ \alpha - \beta } g ) \right\|_{L^2}
\cdot
\left\| \langle \cdot \rangle^M ( \partial^{ \beta } \varphi ) \right\|_{L^2}
\\
&\sim
\langle x - \ell \rangle^{-M},
\end{align*}
since $g ,\varphi \in \mathcal{S}(\mathbb{R}^n)$.
Thus, by the Peetre inequality 
$\langle x \rangle^s \lesssim \langle \ell \rangle^s \cdot \langle x - \ell \rangle^{ | s | } $,
we have for any $M\geq0$
\begin{align*}
\langle x \rangle^{s} \left| V_g f (x,\xi) \right|
&\lesssim
\langle x \rangle^{s} \langle \xi \rangle^{-N} \sum_{\ell\in\mathbb{Z}^n} | c_\ell | \langle x - \ell \rangle^{-M} 
\\
&\lesssim
\langle \xi \rangle^{-N}\sum_{\ell\in\mathbb{Z}^n} \langle \ell \rangle^s | c_\ell | \langle x - \ell \rangle^{-M + |s|}.
\end{align*}
Substitute this estimate into the right hand side of (\ref{equiv ass}) and choose $N\geq0$ such that $N >n/p$.
Then, we have
\begin{align}
\label{seq right}
\begin{split}
\left\| f \right\|_{ M_{q,p}^{(s, 0)} } 
&= 
\left\| 
	\left\| 
		\langle x \rangle^{s} V_g f (x,\xi) 
	\right\|_{L^q (\mathbb{R}^n_x)} 
\right\|_{L^p (\mathbb{R}^n_\xi)} 
\\
&\lesssim
\left\| 
	\left\| 
		\langle \xi \rangle^{-N} \sum_{\ell\in\mathbb{Z}^n} 
		\langle \ell \rangle^s | c_\ell | \langle x - \ell \rangle^{-M + |s|} 
	\right\|_{L^q (\mathbb{R}^n_x)} 
\right\|_{L^p (\mathbb{R}^n_\xi)} 
\\
&\sim
\left\| 
	\sum_{\ell\in\mathbb{Z}^n} \langle \ell \rangle^s | c_\ell | \langle x - \ell \rangle^{-M + |s|} 
\right\|_{L^q (\mathbb{R}^n_x)}. 
\end{split}
\end{align}
Moreover, since $\mathbb{R}^n$ is decomposed by $\mathbb{R}^n = \bigcup_{m \in \mathbb{Z}^n} \widetilde Q_m $ with $ \widetilde Q_m = m + [-1/2,1/2)^n$, the last quantity in (\ref{seq right}) is expressed as
\begin{align*}
\left\{ 
	\sum_{m\in\mathbb{Z}^n} 
	\int_{\widetilde Q_m} 
	\left( 
		\sum_{\ell\in\mathbb{Z}^n} \langle \ell \rangle^s | c_\ell | \langle x - \ell \rangle^{-M + |s|} 
	\right)^q dx 
\right\}^{1/q} .
\end{align*}
Hence, by the fact that $\langle x - \ell \rangle \sim \langle m - \ell \rangle$ for all $x \in \widetilde Q_m$,
we have
\begin{align*}
\left\| f \right\|_{ M_{q,p}^{(s, 0)} } 
\lesssim
\left\{ 
	\sum_{m\in\mathbb{Z}^n} 
	\left( 
		\sum_{\ell\in\mathbb{Z}^n} \langle \ell \rangle^s | c_\ell | \langle m - \ell \rangle^{-M + |s|} 
	\right)^q
\right\}^{1/q} .
\end{align*}
Choosing $M\geq0$ such that $M > n + |s|$ 
and using the convolution relation $\ell^q \ast \ell^1 \hookrightarrow \ell^q$ for $1 \leq q \leq \infty$,
we obtain
\begin{equation}
 \label{conc prop 2}
\left\| f \right\|_{ M_{q,p}^{(s, 0)} } 
\lesssim 
\left( 
	\sum_{\ell\in\mathbb{Z}^n} \langle \ell \rangle^{sq} | c_\ell |^q 
\right)^{1/q} .
\end{equation}

\noindent{\bf Step 3:} Combing (\ref{conc prop 1}) and (\ref{conc prop 2}) with the assumption, we have
\begin{align*} 
\left( \sum_{ m \in \mathbb{Z}^n} |c_m |^p \right)^{1/p}
&\lesssim
\left\| e^{ \mp i | \cdot |^2 } \widehat f \right\|_{ M_{q,p}^{(0, 0)} } 
\sim 
\left\| e^{ \pm i \Delta} f \right\|_{ W_{p,q} }
\\
&\lesssim
\left\| f \right\|_{ W_{p,q}^{s} } 
\sim
\left\| \widehat f \right\|_{ M_{q,p}^{(s, 0)} }
\lesssim 
\left( \sum_{\ell\in\mathbb{Z}^n} \langle \ell \rangle^{sq} | c_\ell |^q \right)^{1/q} ,
\end{align*} 
which is the desired result.
\end{proof}

\section{Proof of the main theorem} \label{sec4}

In this section, we prove Theorem \ref{main thm}. 
As mentioned in Section \ref{sec1}, the ``IF'' part in Theorem \ref{main thm} was already proved 
by Cunanan and Sugimoto \cite{cunanan sugimoto 2014}
and also the ``ONLY IF'' part for $p = \infty$ and $1 \leq q < \infty$ 
in Theorem \ref{main thm} was proved by Cordero and Nicola \cite{cordero nicola 2010},
so that the main contribution of this paper is 
to show the ``ONLY IF'' part for the remaining cases of $p$ and $q$.
However, for the reader's convenience, 
we will prove the ``IF'' and ``ONLY IF'' parts for the whole range $1\leq p,q \leq \infty$.

\begin{proof}[\bf Proof of the ``IF'' part of Theorem \ref{main thm}]
We divide the proof into the following three cases:
  \begin{center}
    \begin{tabular}{lll}
    (a) $1\leq p = q \leq \infty$; & (b) $1 \leq p < q \leq \infty $; & (c) $1 \leq q < p \leq \infty $.
    \end{tabular}
  \end{center}

\noindent{\bf Case (a):} 
Since $M_{p,p} = W_{p,p} $, Theorem A gives the conclusion in this case.

\noindent{\bf Case (b):} 
Recall from Proposition \ref{basic emb} (1)
that $W_{p,p} \hookrightarrow W_{p,q}$ for $p < q$,
and from Proposition \ref{basic emb} (2) 
that $W_{p,q}^s \hookrightarrow W_{p,p}$ for $p < q$ and $s > n (1/p - 1/q)$.
Then, Theorem A combined with $ M_{p,p} = W_{p,p} $ gives
\begin{equation*}
\left\| e^{i\Delta} f \right\|_{ W_{p,q} } 
\lesssim
\left\| e^{i\Delta} f \right\|_{ W_{p,p} }
\lesssim
\left\| f \right\|_{ W_{p,p} }
\lesssim
\left\| f \right\|_{ W_{p,q}^s }.
\end{equation*}

\noindent{\bf Case (c):} 
As above, using Proposition \ref{basic emb} and Theorem A, we have the desired conclusion.
\end{proof}

\begin{proof}[\bf Proof of the ``ONLY IF'' part of Theorem \ref{main thm}]
We prove the ``ONLY IF'' part of Theorem \ref{main thm} by considering the following four cases:
\begin{center}
	\begin{tabular}{llll}
	(a) $1 \leq p = q \leq \infty$; & (b) $1 \leq p < q < \infty $; & (c) $1 < q < p \leq \infty $; &(d) otherwise.
	\end{tabular}
\end{center}

\noindent{\bf Case (a):} 
The relation $W_{p,p}^s = M_{p,p}^s$ and Theorem A complete the proof for this case.

\noindent{\bf Case (b):} 
Note that the condition $1 \leq p < q < \infty$ implies that $1 < q/p < \infty$.
Using (\ref{seq ck}) with $c_k = \langle k \rangle^{-s} |d_k|^{1/p}$ 
for a given finitely supported sequence $\{ d_k \}_{k\in\mathbb{Z}^n}$, we have
\begin{equation} \label{seq dk}
\sum_{k\in\mathbb{Z}^n} \langle k \rangle^{-sp} |d_k| 
\lesssim 
\left( \sum_{k\in\mathbb{Z}^n} | d_k |^{q/p} \right)^{p/q}.
\end{equation}
We take the supremum over $\{ d_k \}_{k\in\mathbb{Z}^n}$ such that $\| d_k \|_{\ell^{q/p}} = 1$. 
Then, we have by (\ref{seq dk})
\begin{equation*}
\left\| \langle k \rangle^{- s p } \right\|_{\ell^{(q/p)^\prime}} 
= 
\sup_{\| d_k \|_{\ell^{q/p}} = 1} \left| \sum_{k\in\mathbb{Z}^n} \langle k \rangle^{-sp} d_k \right| 
\lesssim 
1,
\end{equation*}
which yields that $(q/p)^\prime s p > n$, namely, $s > n(1/p-1/q) = n | 1/p - 1/q | $.

\noindent{\bf Case (c):} 
By duality, our assumption implies $e^{ - i \Delta} : W_{p^\prime,q^\prime} \to W^{-s}_{p^\prime,q^\prime} $, 
which is equivalent from Proposition \ref{lift op} that $e^{ - i \Delta} : W^s_{p^\prime,q^\prime} \to W_{p^\prime,q^\prime} $. 
Note that $1 < q < p \leq \infty $ implies that $1 \leq p^\prime < q^\prime < \infty $ 
and thus $1 < q^\prime / p^\prime < \infty $. 
Then, repeating the proof above, we obtain $(q^\prime / p^\prime )^\prime s p^\prime > n$, 
that is, $s > - n(1/p-1/q) = n | 1/p - 1/q | $.

\noindent{\bf Case (d):} We note that the remaining parts are, in detail, 
\begin{center}
	\begin{tabular}{ll}
	(d-1) $1 \leq p < \infty$ and $q = \infty$; 
	& (d-2) $1 < p \leq \infty $ and $ q = 1$.
	\end{tabular}
\end{center}
We first consider the case (d-1). 
We assume towards a contradiction that $s \leq n | 1/p - 1/q | = n/p$.
Note that we have $e^{ i \Delta}: W^s_{p,\infty} \to W_{p,\infty} $ by the assumption,
which implies that $e^{ i \Delta}: W^{ n/p }_{p,\infty} \to W_{p,\infty}$,
and $e^{ i \Delta}: W_{p,p} \to W_{p,p} $ by Theorem \ref{main thm} (1).
Interpolation yields 
for arbitrary $0 < \theta < 1$ 
that $e^{ i \Delta}: W^{\widetilde s}_{ p ,\widetilde q } \to W_{ p ,\widetilde q} $, 
where $1/ \widetilde q = \theta/p$ and $\widetilde s = (1-\theta)n/p$.
Now, we remark that $\widetilde s = n(1/p-1/\widetilde q)$ and $1 \leq p < \widetilde q  < \infty$.
However, this is a contradiction, 
since we already proved in Case (b) that, for $1 \leq p < \widetilde q  < \infty$, 
$e^{ i \Delta}: W^{\widetilde s}_{ p ,\widetilde q } \to W_{ p ,\widetilde q} $ holds 
only if $\widetilde s > n(1/p-1/\widetilde q)$. 
Thus, we obtain for $1 \leq p < \infty$, 
$e^{ i \Delta}: W^{s}_{p,\infty} \to W_{p,\infty}$ holds only if $s > n / p$. 
The case (d-2) can be also given by a similar argument, so that we omit the detail.
\end{proof}

\section*{Acknowledgments}
The authors appreciate Professor Mitsuru Sugimoto, 
who listened our talk and gave them valuable advice and comments.
The first author is supported by Grant-in-Aid for JSPS Research Fellow (No. 17J00359). 
The second author is partially supported by Grant-in-aid for Scientific Research from JSPS (No. 16K05201).

\end{document}